\documentclass[a4paper]{amsart}

\usepackage{amsfonts, amsmath, amsthm, amssymb, pifont}
\usepackage{enumitem}
\usepackage{upref}
\usepackage[all]{xy}
\usepackage[bookmarksdepth=2]{hyperref}
\usepackage{mathabx}

\theoremstyle{plain}
\newtheorem{thm}{Theorem}[section]

\newtheorem{lem}[thm]{Lemma}

\newtheorem{cor}[thm]{Corollary}
\newtheorem{prop}[thm]{Proposition}

\theoremstyle{definition}
\newtheorem{defn}[thm]{Definition}
\newtheorem{exam}[thm]{Example}

\theoremstyle{remark}
\newtheorem{rem}[thm]{Remark}

\newcommand{\iso}{\cong}
\newcommand{\ZZ}{\mathbb{Z}}
\newcommand{\RR}{\mathbb{R}}
\newcommand{\weyl}{\Delta}
\newcommand{\dweyl}{\nabla}

\newcommand{\ident}{{\rm id}}
\newcommand{\cat}{\mathcal{A}}
\newcommand{\catdash}{\mathcal{B}}
\newcommand{\Lder}{\mathbf{L}}
\newcommand{\Rder}{\mathbf{R}}
\newcommand{\weakequiv}{\mathcal{W}}
\newcommand{\cofibr}{\mathcal{C}}
\newcommand{\fibr}{\mathcal{F}}
\newcommand{\modcat}[1]{{#1}{\rm -mod}}
\newcommand{\grmodcat}[1]{{}_{\rm gr}{#1}{\rm -mod}}
\newcommand{\filtmod}[1]{{}_F{#1}{\rm -mod}}
\newcommand{\radfiltmod}[1]{{}_J{#1}{\rm -mod}}
\newcommand{\injmap}[1]{{#1}{\rm -inj}}
\newcommand{\projmap}[1]{{#1}{\rm -proj}}
\newcommand{\cofmap}[1]{{#1}{\rm -cof}}
\newcommand{\fibmap}[1]{{#1}{\rm -fib}}
\newcommand{\derivcat}[1]{D({#1})}
\newcommand{\Homfilt}{\Hom_{F}}
\newcommand{\Homradfilt}{\Hom_{J}}
\newcommand{\Homgr}{\Hom_{\rm gr}}
\newcommand{\Extfilt}{\Ext_{F}}
\newcommand{\Extradfilt}{\Ext_{J}}
\newcommand{\Extgr}{\Ext_{\rm gr}}
\newcommand{\fl}{\mathrm{fl}}

\DeclareMathOperator{\Ext}{Ext}
\DeclareMathOperator{\Hom}{Hom}
\DeclareMathOperator{\End}{End}
\DeclareMathOperator{\soc}{soc}
\DeclareMathOperator{\rad}{rad}
\DeclareMathOperator{\head}{head}
\DeclareMathOperator{\chmod}{{\bf Ch}}
\DeclareMathOperator{\grees}{Rees}
\DeclareMathOperator{\im}{im}
\DeclareMathOperator{\Mor}{Mor}
\DeclareMathOperator{\Ho}{Ho}

\DeclareMathOperator{\pr}{pr}

\begin{document}

\title[Rigidity of tilting modules]{Radically filtered quasi-hereditary algebras and rigidity of tilting modules}
\author{Amit Hazi}
\date{29 April 2015}
\address{Department of Pure Mathematics and Mathematical Statistics\\
Centre for Mathematical Sciences\\
University of Cambridge\\
Wilberforce Road\\
Cambridge\\
CB3 0WB \\
United Kingdom}
\email{A.Hazi@dpmms.cam.ac.uk}
\subjclass[2010]{16D70 (Primary); 16D90, 20G40 (Secondary)}

\begin{abstract}
Let $A$ be a quasi-hereditary algebra. We prove that in many cases, a tilting module is rigid (i.e.~has identical radical and socle series) if it does not have certain subquotients whose composition factors extend more than one layer in the radical series or the socle series. We apply this theorem to give new results about the radical series of some tilting modules for $SL_4(K)$, where $K$ is a field of positive characteristic.
\end{abstract}

\maketitle

\tableofcontents

\section*{Introduction}

Let $A$ be a finite-dimensional quasi-hereditary algebra, with standard modules $\weyl(\lambda)$ and costandard modules $\dweyl(\lambda)$. The tilting modules for $A$ were first characterized by Ringel in \cite{ringel} as modules with both standard and costandard filtrations. The goal of this paper is to describe when tilting modules are rigid (i.e.~have identical radical and socle series). The paper can be split roughly into two parts. In the first part, we describe filtered algebras and the machinery for working with them in a derived setting. In the second part, we use this machinery to prove our rigidity results, which we apply to calculating the Loewy structure of some tilting modules.

Our work was partially inspired by the work of Bowman, Doty, and Martin \cite{bdm-small,bdm-large} which described the indecomposable summands of the tensor product $L \otimes L'$ of two irreducible $SL_3(K)$ modules, where $K$ is a field of positive characteristic. For a general reductive algebraic group $G$, the category of rational $G$-modules is a highest-weight category, which is closely related to the notion of a quasi-hereditary algebra \cite{donkin}. This means that tilting modules can be defined for algebraic groups using Ringel's classification. In particular, tilting modules for algebraic groups naturally appear as some of the indecomposable summands of $L \otimes L'$. 

With few exceptions, the tilting modules in \cite{bdm-small,bdm-large} (and in a previous paper \cite{doty-henke} on the $SL_2(K)$ case) are all rigid. Andersen and Kaneda showed why this is the case by proving a rigidity result for tilting modules for quantum groups and algebraic groups in positive characteristic \cite{andersen-kaneda}. They showed that tilting modules above the Steinberg weight which are not ``too close'' to the walls of the dominant chamber or ``too high'' in the case of algebraic groups are rigid.

We had hoped to use our rigidity result as a stepping-stone for similar tensor decomposition work for $SL_4(K)$. In the last section, we do succeed in showing that the restricted tilting modules are rigid and calculate their Loewy structures (a new result as far as we are aware). The calculations rely heavily on knowledge of the Weyl module structures, which can be difficult to compute in general. Further work in this direction seems necessary for this method to be extended to higher weight tilting modules.

\section{Filtered algebras}

Throughout this paper, $A$ denotes a finite-dimensional algebra over a field $K$.

\begin{defn}
A generalized filtration on $A$ is a collection of $K$-subspaces $F^i A$ (indexed by integers $i$) such that the $K$-linear span of $\{F^i A\}$ is $A$, $1 \in F^0 A$, and $(F^i A)(F^j A) \subseteq F^{i+j} A$ for all $i,j$.
\end{defn}

This is similar to the notion of an ascending or descending filtration on $A$, but without the containment condition. If $A$ has a generalized filtration $F^\bullet$ we call $A$ a generalized filtered algebra. In this paper we will often omit ``generalized'' for brevity.

\begin{defn}
\hfill
\begin{itemize}
\item A filtered module over a filtered algebra $A$ is an $A$-module $M$ equipped with a collection of $K$-subspaces $F^i M$ indexed over the integers such that the $K$-linear span of $\{F^i M\}$ is $M$ and $(F^i A)(F^j M) \subseteq F^{i+j} M$ for all $i,j$.

\item A homomorphism between filtered $A$-modules $M$ and $M'$ with filtrations $F^\bullet$ and $F^{'\bullet}$ is an $A$-module homomorphism $f:M \rightarrow M'$ such that $f(F^i M) \subseteq F^{'i} M'$ for all $i$.
\end{itemize}
\end{defn}

If $M$ is a filtered $A$-module and $M' \leq M$ is an $A$-module, then there are natural filtrations on $M'$ and $M/M'$ making them into filtered modules, namely $F^i M'=F^i M \cap M'$ and $F^i(M/M')=(F^i M+M')/M'$. Combining these two constructions, we can give any subquotient $M'/M''$ of $M$ the filtration
\begin{equation*}
F^i(M'/M'')=(F^i M \cap M'+M'')/M''
\end{equation*}
by first considering $M'$ as a submodule of $M$ and then considering $M'/M''$ as a quotient of $M'$. This is well-defined, for if we apply these processes in the opposite order, we get 
\begin{align*}
F^i(M/M'')& =(F^i M+M'')/M'' \\
F^i(M'/M'')& =((F^i M+M'') \cap M')/M'' \\
& =(F^i M \cap M'+M'')/M''
\end{align*}
which gives the same filtration.

We write $\filtmod{A}$ for the category of filtered modules over a filtered algebra $A$. This category is always additive and in fact pre-abelian, yet even in the case of ascending/descending filtrations, $\filtmod{A}$ is not necessarily abelian.

\begin{exam}
\label{exam:rad-filt}
Let $J(A)$ be the Jacobson radical of $A$, and define the filtration $J^i A=J(A)^i$ for $i \geq 0$ and $J^i A=A$ for $i<0$. This gives $A$ a (descending) filtered structure, and any $A$-module $M$ can be given a filtration $J^i M=J(A)^i M=\rad^i M$ (and $J^i M=M$ for $i<0$) which is compatible with the filtration on $A$. In this case, we write $\radfiltmod{A}$ for the filtered module category.
\end{exam}

\section{Model categories}

In order to define a functor analogous to $\Ext$ on $\filtmod{A}$ it will be necessary to use some technology from homotopy theory, which we describe below. The primary reference for this section is \cite[Chapter 1]{hovey}. Throughout this section, $\cat$ and $\catdash$ denote arbitrary categories.

\subsection{Model structures}

\begin{defn}
Suppose $i:U \rightarrow V$ and $p:X \rightarrow Y$ are maps in a category $\cat$. Then $i$ has the left lifting property with respect to $p$ and $p$ has the right lifting property with respect to $i$ if for every commutative diagram of the following form
\begin{equation*}
\xymatrix{
U \ar[d]_{i} \ar[r]^{f} & X \ar[d]^{p} \\
V \ar[r]_{g} & Y
}
\end{equation*}
there exists a map $h:V \rightarrow X$ such that two triangles introduced in the above diagram commute, i.e.~$hi=f$ and $ph=g$.
\end{defn}

In this situation we write $i \boxslash p$. A map $h$ fitting into such a commutative square is called a lift.

\begin{defn}
\label{defn:modcat}
A model structure on a category $\cat$ is a collection of three subclasses $\weakequiv,\cofibr,\fibr$ of $\Mor \cat$ which satisfy the following properties:
\begin{enumerate}[label=(\roman*)]
\item (2-out-of-3) Suppose $u,v \in \Mor \cat$ such that $vu$ is defined. If two of $u$, $v$, and $vu$ are in $\weakequiv$ then so is the third.

\item (Retracts) Given a commutative diagram of the following form
\begin{equation*}
\xymatrix{
U \ar[d]_{u} \ar[r] \ar@/^/[rr]^{\ident} & C \ar[d]_{v} \ar[r] & A \ar[d]^{u} \\
V \ar[r] \ar@/_/[rr]_{\ident} & D \ar[r] & B
}
\end{equation*}
if $v$ is in $\weakequiv$, $\cofibr$, or $\fibr$ then so is $u$.

\item (Lifting) Using the obvious setwise extension of the symbol $\boxslash$, we have $(\weakequiv \cap \cofibr) \boxslash \fibr$ and $\cofibr \boxslash (\weakequiv \cap \fibr)$.

\item (Factorization) For every $f \in \Mor \cat$, there exist two factorizations:
\begin{itemize}
\item $f=pi$ where $i \in \weakequiv \cap \cofibr$ and $p \in \fibr$,

\item $f=qj$ where $j \in \cofibr$ and $q \in \weakequiv \cap \fibr$.
\end{itemize}
\end{enumerate}
\end{defn}

A map in one of $\weakequiv$, $\cofibr$, or $\fibr$ is called a weak equivalence, cofibration, or fibration respectively. A map in $\weakequiv \cap \cofibr$ or $\weakequiv \cap \fibr$ is called a trivial cofibration or a trivial fibration respectively. In categories with initial and terminal objects (denoted $0$ and $1$ respectively), an object $X$ of $\cat$ is called cofibrant if $0 \rightarrow X$ is a cofibration or fibrant if $X \rightarrow 1$ is a fibration.

Sometimes a distinction is made between a ``category with model structure'' and a so-called ``model category.'' A model category is simply a category with a model structure which contains all finite limits and colimits. A closed model category is a model category which additionally contains all small limits and colimits. Since the categories we will be using later have all such limits, we will freely use the phrase ``model category'' instead of ``category with model structure.''

\subsection{Homotopy categories and derived functors}

The primary motivation for model structures is the homotopy category (sometimes also called the derived category). The homotopy category of a model category is a generalization of the classical derived category $\derivcat{\modcat{A}}$ obtained from the category of cochain complexes $\chmod(\modcat{A})$. Namely, the homotopy category is obtained by adding the inverses of certain ``equivalences'' to the original category. One can think of model categories as categories with just enough structure to enable calculations in homotopy categories.

\begin{defn}
Let $\cat$ be a category with a model structure given by $\weakequiv,\cofibr,\fibr$. The homotopy category (or derived category) of $\cat$ is a category $\Ho \cat$ and a functor $\gamma_\cat:\cat \rightarrow \Ho \cat$ which is the localization of $\cat$ at $\weakequiv$. 

In other words, $\gamma_\cat$ maps $\weakequiv$ to isomorphisms, and $\Ho \cat$ is universal with this property in the sense that if another functor $F:\cat \rightarrow \catdash$ maps $\weakequiv$ to isomorphisms, there exists a unique factorization $F=(\Ho F)\gamma_\cat$ for some functor $\Ho F:\Ho \cat \rightarrow \catdash$.
\end{defn}

\begin{defn}
Let $F:\cat \rightarrow \catdash$ be a functor between two model categories. The left derived functor of $F$ is a functor $\Lder F:\Ho \cat \rightarrow \Ho \catdash$ with a natural transformation $\varepsilon:(\Lder F)\gamma_\cat \Rightarrow \gamma_\catdash F$ called the counit which is universal in the following sense: for any other functor $G:\Ho \cat \rightarrow \Ho \catdash$ with a natural transformation $\zeta:G\gamma_\cat \Rightarrow \gamma_\catdash F$, there is a unique $\lambda:G \Rightarrow \Lder F$ such that $\zeta=\varepsilon \circ \lambda\gamma_{\cat}$.
\begin{equation*}
\xymatrix{
\cat \ar[d]_{\gamma_\cat} \ar[r]^{F} & \catdash \ar[d]^{\gamma_\catdash} & \ar@{}[d]|{=} & \cat \ar[d]_{\gamma_\cat} \ar[r]^{F} \ar[ddl]_{\gamma_\cat} & \catdash \ar[d]^{\gamma_\catdash} \\
\Ho \cat \ar[r]_{G} \ar@{=>}[ur]_{\zeta} & \Ho \catdash & & \Ho \cat \ar[r]_{\Lder F} \ar@{=>}[ur]_{\varepsilon} & \Ho \catdash \\
& & \Ho \cat \ar[urr]_{G} \ar@{==>}[ur]|{\lambda\gamma_{\cat}} & &
}
\end{equation*}

Similarly, the right derived functor of $F$ is a functor $\Rder F:\Ho \cat \rightarrow \Ho \catdash$ with a natural transformation $\eta:\gamma_\catdash F \Rightarrow (\Rder F)\gamma_\cat$ called the unit which has the following universal property: for any other functor $G:\Ho \cat \rightarrow \Ho \catdash$ with a natural transformation $\theta:\gamma_\catdash F \Rightarrow G\gamma_\cat$, there is a unique $\mu:\Rder F \Rightarrow G$ such that $\theta=\mu\gamma_{\cat} \circ \eta$.
\begin{equation*}
\xymatrix{
\cat \ar[d]_{\gamma_\cat} \ar[r]^{F} & \catdash \ar[d]^{\gamma_\catdash} \ar@{=>}[dl]_{\theta} & \ar@{}[d]|{=} & \cat \ar[d]_{\gamma_\cat} \ar[r]^{F} \ar[ddl]_{\gamma_\cat} & \catdash \ar[d]^{\gamma_\catdash} \ar@{=>}[dl]_{\eta} \\
\Ho \cat \ar[r]_{G} & \Ho \catdash & & \Ho \cat \ar[r]_{\Rder F} \ar@{==>}[dl]|{\mu\gamma_{\cat}} & \Ho \catdash \\
& & \Ho \cat \ar[urr]_{G}  & &
}
\end{equation*}
\end{defn}

In general, calculating derived functors can be difficult if no extra information about the functor is given. Thus we will restrict ourselves to taking derived functors of functors which preserve some aspects of the model structure.

\begin{defn}
Let $\cat$ and $\catdash$ be two model categories.
\begin{itemize}
\item A left Quillen functor $F:\cat \rightarrow \catdash$ is a functor that is left adjoint and preserves cofibrations and trivial cofibrations.

\item A right Quillen functor $G:\catdash \rightarrow \cat$ is a functor that is right adjoint and preserves fibrations and trivial fibrations.

\item A Quillen adjunction $F \dashv G:\cat \leftrightarrows \catdash$ is an adjunction where $F$ is a left Quillen functor and $G$ is a right Quillen functor.
\end{itemize}
\end{defn}

The following proposition shows that these definitions are overdetermined.

\begin{prop}[\cite{zllow}]
\label{prop:quillfuncttfae}
Let $F \dashv G:\cat \leftrightarrows \catdash$ be an adjunction between two model categories. The following are equivalent.
\begin{enumerate}[label=(\roman*)]
\item $F \dashv G$ is a Quillen adjunction.

\item $F$ preserves cofibrations and trivial cofibrations.

\item $G$ preserves fibrations and trivial fibrations.

\item $F$ preserves cofibrations and $G$ preserves fibrations.
\end{enumerate}
\end{prop}

If $F$ is a Quillen functor, then the derived functor of $F$ can be calculated via a process called (co)fibrant replacement. Suppose a category $\cat$ with model structure has initial and terminal objects $0,1$. For any object $X$, we can factor the map $0 \rightarrow X$ as a map $0 \rightarrow QX \xrightarrow{q_X} X$, where $0 \rightarrow QX$ is a cofibration (and thus $QX$ is cofibrant) and $QX \xrightarrow{q_X} X$ is a trivial fibration. This mapping $X \mapsto QX$ defines a functor\footnote{Functoriality of $Q$ requires that the factorization in Definition \ref{defn:modcat} be functorial. See \cite[1.1.1 (2)]{hovey} for details.} called the cofibrant replacement functor, and $q_X$ defines the components for a natural transformation. Similarly there is a fibrant replacement functor $R$ and a natural trivial cofibration with components $X \xrightarrow{r_X} RX$.

\begin{prop}[\cite{hovey},\cite{zllow}]
If $F:\cat \rightarrow \catdash$ is a left Quillen functor, the left derived functor of $F$ exists, and can be calculated as the following composition:
\begin{equation*}
\xymatrix{
\Ho \cat \ar[r]^{\Ho \gamma_{\cat}Q} & \Ho \cat_c \ar[r]^{\Ho \gamma_{\catdash}F} & \Ho \catdash
}
\end{equation*}
where $\Ho \cat_c$ denotes the full subcategory of cofibrant objects in $\Ho \cat$.
\end{prop}

For calculating the right derived functor of a right Quillen functor, we use the fibrant replacement functor in a similar way.

Finally Quillen adjunctions have the property that they induce adjunctions in the derived categories, as described below.

\begin{thm}[{\cite[1.3.10]{hovey}}]
If $F\dashv G:\cat \leftrightarrows \catdash$ is a Quillen adjunction, then $\Lder F,\Rder G:\Ho \cat \leftrightarrows \Ho \catdash$ are also adjoint functors. This adjunction is called the derived adjunction of $F \dashv G$.
\end{thm}

\subsection{Some examples}

We will first describe perhaps the most well-known model category, the category of cochain complexes of an abelian category. Let $\cat$ denote the abelian category $\modcat{A}$ for some algebra $A$, and $\chmod{\cat}$ the category of cochain complexes over $\cat$. The first step is describing what projective or injective relative to a class of morphisms means.

\begin{defn}
Let $I$ be a subclass of maps in some category $\cat$.
\begin{itemize}
\item $\injmap{I}=\{f \in \Mor \cat \mid I \boxslash f\}$

\item $\projmap{I}=\{f \in \Mor \cat \mid f \boxslash I\}$

\item $\cofmap{I}=\projmap{(\injmap{I})}$

\item $\fibmap{I}=\injmap{(\projmap{I})}$
\end{itemize}
\end{defn}

\begin{exam}
Define the following complexes $S^n$ and $D^n$ in $\chmod{\cat}$
\begin{align*}
(S^n)^k& =\begin{cases}
A & \text{if $k=n$} \\
0 & \text{otherwise}
\end{cases} & 
(D^n)^k& =\begin{cases}
A & \text{if $k=n,n+1$} \\
0 & \text{otherwise}
\end{cases}
\end{align*}
where all differentials of $S^n$ are $0$, and the only non-trivial differential map of $D^n$ is $d^n:A \xrightarrow{\ident} A$. For each $n \in \ZZ$ we have an injection $S^{n+1} \rightarrow D^n$ given by the identity in (homological) degree $n+1$ and $0$ elsewhere. Let
\begin{align*}
I& =\{S^{n+1} \rightarrow D^n \mid n \in \ZZ\} \\
J& =\{0 \rightarrow D^n \mid n \in \ZZ\} \\
\weakequiv& =\{f:X \rightarrow Y \mid \text{$H^n(f)$ is an isomorphism for all $n \in \ZZ$}\}
\end{align*}
Here $H^n(f)$ denotes the homomorphism on cohomology groups induced by a cochain map. In other words, $\weakequiv$ consists of the set of quasi-isomorphisms in $\chmod{\cat}$.

\begin{thm} \label{thm:chmod-proj-model}
Let $\cofibr=\cofmap{I}$ and $\fibr=\injmap{J}$. Then the sets $\weakequiv,\cofibr,\fibr$ define a model structure called the projective model structure on $\chmod{A}$.
\end{thm}
\begin{proof}
See, for example, \cite[2.3]{hovey} or \cite[1.2]{di-natale}.
\end{proof}

The fibrations in this model structure are the degreewise surjective cochain maps, and all complexes are fibrant. A cofibrant complex $X$ has the property that for each $n$, $X^n$ is a projective $A$-module. For bounded above complexes, the converse is also true, but unbounded cofibrant complexes are trickier to understand. The cofibrations are the degreewise split injective cochain maps with cofibrant cokernels. Throughout this paper we will use the abbreviation $\derivcat{\cat}$ for $\Ho \chmod \cat$.
\end{exam}

Here is another example of how one can extend this model structure to similar-looking categories.

\begin{exam}
Suppose $B$ is a graded $K$-algebra, i.e.~$B=\bigoplus_i B_i$ with $1 \in B_0$ and $B_i B_j \subseteq B_{i+j}$. Let $\catdash=\grmodcat{B}$, the category of graded $B$-modules. The category $\chmod{\catdash}$ of cochain complexes of graded modules has a projective model structure very similar to the one above.

Let $S^n$ and $D^n$ take the obvious gradings from $B$:
\begin{align*}
((S^n)^k)_i& =\begin{cases}
B_i & \text{if $k=n$} \\
0 & \text{otherwise}
\end{cases} & 
((D^n)^k)_i& =\begin{cases}
B_i & \text{if $k=n,n+1$} \\
0 & \text{otherwise}
\end{cases}
\end{align*}
The differentials are all graded homomorphisms as they are all $0$ or $\ident$.

For a graded $B$-module $M$ and $r \in \ZZ$ define the grading shift $M(r)_i=M_{i-r}$. It is easy to see that shifting is functorial on $\catdash$ and $\chmod{\catdash}$.

Now we define
\begin{align*}
I_{\rm gr}& =\{S^{n+1}(r) \rightarrow D^n(r) \mid n,r \in \ZZ\} \\
J_{\rm gr}& =\{0 \rightarrow D^n(r) \mid n,r \in \ZZ\} \\
\weakequiv_{\rm gr}& =\{f:X \rightarrow Y \mid \text{$H^n(f)$ is an isomorphism for all $n,i \in \ZZ$}\}
\end{align*}

\begin{thm}
Let $\cofibr_{\rm gr}=\cofmap{I_{\rm gr}}$ and $\fibr_{\rm gr}=\injmap{J_{\rm gr}}$. Then the sets $\weakequiv_{\rm gr},\cofibr_{\rm gr},\fibr_{\rm gr}$ define a model structure called the projective model structure on $\chmod{\catdash}$.
\end{thm}

\begin{proof}
Adapt the proof of Theorem \ref{thm:chmod-proj-model} to the graded case. This is especially easy because $\grmodcat{B}$ is an abelian category like $\modcat{A}$ so kernels, images, cokernels, etc.~all make sense.
\end{proof}

Again the fibrations in this model structure are the homological degreewise surjective cochain maps, and all complexes are fibrant. A bounded above complex $X$ is cofibrant if and only if $X^n$ is projective as a graded $B$-module for all $n$. The cofibrations are the degreewise split injective cochain maps with cofibrant cokernels.
\end{exam}

\section{Filtered cochain complexes}

Suppose $A$ is a filtered algebra, and let $\cat=\filtmod{A}$. Using the examples from the previous section, we define a model structure on $\chmod{\cat}$ following \cite{di-natale}.

\subsection{Model structure}

Define the following filtrations on $S^n$ and $D^n$ defined above:
\begin{align*}
F^i (S^n)^k& =\begin{cases}
F^i A & \text{if $k=n$} \\
0 & \text{otherwise}
\end{cases} & 
F^i (D^n)^k& =\begin{cases}
F^i A & \text{if $k=n,n+1$} \\
0 & \text{otherwise}
\end{cases}
\end{align*}
It is easy to verify that the differentials are all homomorphisms of filtered modules. 

Now for a filtered $A$-module $M$ and $r \in \ZZ$ define the filtration shift $F^i (M\langle r\rangle)=F^{i-r} M$. It is evident that $M\langle r\rangle$ is still a filtered module, and that shifting is functorial on $\cat$ and $\chmod{\cat}$.

In this vein we define
\begin{align*}
I_F& =\{S^{n+1}\langle r\rangle \rightarrow D^n\langle r\rangle \mid n,r \in \ZZ\} \\
J_F& =\{0 \rightarrow D^n\langle r\rangle \mid n,r \in \ZZ\} \\
\weakequiv_F& =\{f:X \rightarrow Y \mid \text{$H^n(F^i f)$ is an isomorphism for all $n,i \in \ZZ$}\}
\end{align*}
In other words, $\weakequiv_F$ consists of the set of filtration-wise quasi-isomorphisms in $\chmod{\cat}$.
\begin{thm}
Let $\cofibr_F=\cofmap{I_F}$ and $\fibr=\injmap{J_F}$. Then the sets $\weakequiv_F,\cofibr_F,\fibr_F$ define a model structure called the projective model structure on $\chmod{\cat}$.
\end{thm}

\begin{proof}
See \cite[1.3]{di-natale} for a full proof in the case when $A$ has the trivial filtration ($F^i A=A$ for $i \geq 0$). This is an adaptation of the proof of Theorem \ref{thm:chmod-proj-model} but with extra care for filtration degrees.  The general proof is essentially identical.
\end{proof}

As expected, the fibrations in this model structure are the (homological and filtration) degreewise surjective cochain maps, and all complexes are fibrant. A bounded below complex $X$ is cofibrant if and only if $X^n$ is projective as a filtered $A$-module for all $n$ (we explain what this means in greater detail in \ref{ss:filtproj}). The cofibrations are the degreewise split injective cochain maps with cofibrant cokernels.

\subsection{The Rees algebra}

Now we consider connections to the algebra
\begin{equation*}
B=\grees{A}=\bigoplus_{i \in \ZZ} (F^i A)t^i
\end{equation*}
which is a subalgebra of $A[t]$. It has a grading induced both by the grading on $A[t]$ and the filtration structure on $A$. Functionally the indeterminate $t$ does nothing but record the grading, so that $at^i$ is distinct from $at^j$ in $\grees{A}$ for any $a \in F^i A \cap F^j A$. Let $\catdash=\grmodcat{B}=\grmodcat{(\grees{A})}$. It is clear that the $\grees$ construction is functorial, i.e.~$\grees:\cat \rightarrow \catdash$ is a functor mapping a filtered module $M$ to the graded $\grees(A)$-module
\begin{equation*}
\grees{M}=\bigoplus_i (F^i M)t^i
\end{equation*}

\begin{thm}
The functor $\grees$ has a left adjoint $\varphi:\catdash \rightarrow \cat$. The module structure on $\varphi(M)$ is the quotient $M/LM$ where $L$ is the two-sided ideal of $\grees{A}$ generated by
\begin{equation*}
\left\{ \sum_i a_i t^i \mathrel{}\middle\vert\mathrel{} a_i \in F^i A,\ \sum_i a_i=0 \right\}
\end{equation*}
The filtration on $\varphi(M)$ is given by defining $F^i M$ to be the image of $M_i$ in this quotient.
\end{thm}

\begin{proof}
First we show that $\varphi$ is a well-defined functor. This amounts to showing that $(\grees{A})/L \iso A$ so that $M/LM$ has a natural $A$-module structure. There is a natural homomorphism  of ordinary modules
\begin{align*}
\grees{A} & \longrightarrow A \\
a_i t^i & \longmapsto a_i
\end{align*}
and the kernel is clearly $L$. Also, it is surjective because the span of $\{F^i A\}$ is $A$. For the filtration, note that the span of the images of $M_i$ in the quotient $M/LM$ clearly span the quotient. Also, if $a_i \in F^i A$ and $m_j \in M_j$, then $a_i(m_j+LM)=a_i t^i(m_j+LM) \in M_{i+j}+LM$, so this truly gives a filtered $A$-module structure.

To show the adjunction, we show that $\Homfilt(\varphi(M),N) \iso \Homgr(M,\grees{N})$ for $M$ a graded $\grees(A)$-module and $N$ a filtered $A$-module. For $f \in \Homfilt(\varphi(M),N)$, we will define a corresponding $g \in \Homgr(M,\grees{N})$ degreewise in $M$. Suppose $m_i \in M_i$. By the filtration on $\varphi(M), f(m_i+LM) \in f(F^i \varphi(M)) \subseteq F^i N$. So define $g(m_i)=f(m_i+LM)t^i$ and extend linearly. This defines a graded homomorphism as required.

To go the other way, suppose $g \in \Homgr(M,\grees{N})$. For $\overline{m_i} \in F^i \varphi(M)$, pick some $m_i \in M_i$ such that $m_i+LM=\overline{m_i}$. Define $f \in \Homfilt(\varphi(M),N)$ by setting $f(\overline{m_i})=n_i$ if $g(m_i)=n_i t^i$ and extending linearly. To see that this is well-defined, we need to show that $g(LM)=0$. Yet this is clearly true because $g(LM)=Lg(M) \subseteq L\grees{N}=0$ by action of $\grees{A}$ on $\grees{N}$. It is clear that this homomorphism is filtered, and these correspondences are inverse to each other.
\end{proof}

\begin{lem}
The adjunction $\varphi \dashv \grees$ is a Quillen adjunction of model categories, i.e.~$\varphi$ preserves cofibrations and trivial cofibrations while $\grees$ preserves fibrations and trivial fibrations.
\end{lem}

\begin{proof}
First we show that $\grees(\injmap{\varphi(I)}) \subseteq \injmap{I}$ and $\varphi(\cofmap{I}) \subseteq \cofmap{\varphi(I)}$ for an arbitrary class of maps $I$. Suppose $f \in \injmap{\varphi(I)}$ and $g \in I$ such that there is a diagram of the form
\begin{equation*}
\xymatrix{
A \ar[d]_{g} \ar[r]^{} & \grees{X} \ar[d]^{\grees f} \\
B \ar[r]_{} & \grees{Y}
}
\end{equation*}
We need to show this diagram has a lift. By adjointness, we may form the following diagram
\begin{equation*}
\xymatrix{
\varphi(A) \ar[d]_{\varphi(g)} \ar[r]^{} & X \ar[d]^{f} \\
\varphi(B) \ar[r]_{} & Y
}
\end{equation*}
which has a lift $h:\varphi(B) \rightarrow X$. It is easy to see that the corresponding map $h':B \rightarrow \grees{X}$ is a lift for the first diagram. We can abbreviate this argument to one line by abuse of notation and remembering that adjointness works similarly with the symbol $\boxslash$ as it does with $\Hom$: $\varphi(I) \boxslash \injmap{\varphi(I)} \Rightarrow I \boxslash \grees(\injmap{\varphi(I)})$. Similarly, we have 
\begin{align*}
\cofmap{I} \boxslash \injmap{I} & \Rightarrow \cofmap{I} \boxslash \grees(\injmap{\varphi(I)}) \\
& \Rightarrow \varphi(\cofmap{I}) \boxslash \injmap{\varphi(I)} \\
& \Rightarrow \varphi(\cofmap{I}) \subseteq \cofmap{\varphi(I)}
\end{align*}

Now we apply the above to the model categories $\cat$ and $\catdash$. First note that $\varphi(J_\mathrm{gr})=J_F$ and $\varphi(I_\mathrm{gr})=I_F$. Now we have $\grees(\injmap{\varphi(J_\mathrm{gr})})=\grees(\injmap{J_F}) \subseteq \injmap{J_\mathrm{gr}}$, showing that $\grees$ maps fibrations to fibrations. Similarly, $\varphi(\cofmap{I_\mathrm{gr}}) \subseteq \cofmap{\varphi(I_\mathrm{gr})}=\cofmap{I_F}$ so $\varphi$ maps cofibrations to cofibrations. By Proposition \ref{prop:quillfuncttfae}, the adjunction is a Quillen adjunction.
\end{proof}

\subsection{Filtered projective modules}
\label{ss:filtproj}

\begin{defn}
Let $A$ be a filtered algebra. A filtered module $P$ is called (filtered) projective if for any filtration surjective homomorphism $p:M \rightarrow N$ and any homomorphism $g:P \rightarrow N$, there exists a homomorphism $h:P \rightarrow M$ such that $ph=g$.
\end{defn}

There are many reasons for this to be the correct definition of projective in this context, including the following two lemmas.

\begin{lem}
An $A$-module $P$ is filtered projective if and only if it is a summand of a direct sum of (possibly filtration shifted) copies of $A$.
\end{lem}

\begin{proof}
Suppose $P$ is a summand of $L=A\langle r_1\rangle \oplus \dotsb \oplus A\langle r_k\rangle$. Let $g:M \rightarrow N$ be a filtration surjective homomorphism and let $g:P \rightarrow N$ be any homomorphism. Write $q:L \rightarrow P$ for the projection map and $i:P \rightarrow L$ for the inclusion map. Let $n_1,\dotsc,n_k \in N$ be the images of $1$ (in each copy of $A$) under the composite map $gq$. Since the copies of $A$ are filtration shifted we have $n_i \in F^{r_i}N$ for each $i$. Let $m_i \in F^{r_i}M$ such that $p(m_i)=n_i$ for each $i$. There is a unique homomorphism $h':L \rightarrow M$ which maps the $i$th copy of $1$ to $m_i$, so the map $h=h'i$ is a lift and $P$ is projective.

Conversely, suppose $P$ is projective. The module $P$ has a generating set $\{p_i\}$. By writing each generator as the sum of different filtration components, we may assume that each generator $p_i$ is contained in some filtered part $F^{r_i}P$ for integers $r_i$. As above, there is a unique homomorphism $q:L \rightarrow A$ where $L=\oplus_i A\langle r_i \rangle$ mapping the $i$th copy of $1$ to $p_i$. Clearly this map is surjective. If it isn't filtration surjective, suppose there is some $p \in F^r P$ such that $p \notin q(F^r F)$. Then we can add $p$ to the list of generators, replace $L$ with $L \oplus A\langle r\rangle$, and try again. Thus we have a filtration surjective homomorphism $q:L \rightarrow P$. Using projectivity, we show that $q$ has a right inverse $i:P \rightarrow L$ with $pi=\ident_P$.
\end{proof}

\begin{rem}
It doesn't matter if $P$ is a summand as a filtered module or not. If $P$ is a summand of a module $L=\oplus_i A\langle r_i\rangle$ as a module over an ordinary algebra $A$, then $P$ can be given a filtration compatible with the filtration on $F$. Namely, define $F^i P=p(F^i L)$ where $p$ the canonical projection $p:F \rightarrow P$.
\end{rem}

\begin{lem}
If $X$ is a cofibrant cochain complex in $\chmod{A}$ then for each $n \in \ZZ$, $X^n$ is filtered projective. Conversely, if $X$ is a complex which is bounded above such that $X^n$ is filtered projective, then $X$ is cofibrant.
\end{lem}

\begin{proof}
Adapt the proof of the similar fact in \cite[2.3.6]{hovey}. The key fact here is that fibrations in this model structure are filtration surjective, not just surjective.
\end{proof}

\begin{defn}
Let $M$ be a filtered $A$-module. A filtered projective resolution of $M$ consists of a complex $P$ (indexed following the chain complex convention, with $P_n=0$ for $n<0$) and a homomorphism $P_0 \rightarrow M$ such that
\begin{enumerate}[label=(\roman*)]
\item The complex $P$ is filtered exact at each $n>0$, i.e.~$H_n(F^i P)=0$ for all $i$.

\item The homomorphism $P_0 \rightarrow M$ is filtered surjective.
\end{enumerate}
\end{defn}

It is easy to see using the previous lemmas that filtered projective resolutions exist and are cofibrant replacements for complexes concentrated in one homological degree.

\begin{defn}
For two filtered modules $M,N$, define 
\begin{equation*}
\Extfilt(M,N)=\Hom_{\derivcat{\cat}}(\gamma M,\gamma N[i])
\end{equation*}
\end{defn}

\begin{prop}
\label{prop:extfilt-extgr}
For any two filtered $A$-modules $M$ and $N$, we have 
\begin{equation*}
\Extfilt^i(M,N) \iso \Extgr^i(\grees M,\grees N)
\end{equation*}
\end{prop}

\begin{proof}
As $\catdash$ is an abelian category, we know that 
\begin{equation*}
\Extgr^i(\grees M,\grees N)\iso \Hom_{\derivcat{\catdash}}(\gamma\grees M,\gamma\grees N[i])
\end{equation*}
Now use the derived adjunction:
\begin{align*}
\Hom_{\derivcat{\catdash}}(\gamma\grees M,\gamma\grees N[i])& \iso \Hom_{\derivcat{\catdash}}(\gamma\grees M,\Rder\grees \gamma N[i]) \\
& \iso \Hom_{\derivcat{\cat}}(\Lder\varphi\gamma\grees M,\gamma N[i]) \\
& \iso \Hom_{\derivcat{\cat}}((\Ho \gamma\varphi)\circ (\Ho \gamma Q) \circ \gamma\grees M, \gamma N[i]) \\
& =\Hom_{\derivcat{\cat}}((\gamma\varphi Q\grees M, \gamma N[i]) \\
\end{align*}
Now suppose we have a projective resolution $P$ for $M$. As $\grees$ is clearly an additive functor, it maps projective modules to projective modules, since in both cases these are (possibly shifted) summands of the algebra. The map $P_0 \rightarrow M$ induces a trivial fibration $P \rightarrow M$, and as $\grees$ is a right Quillen functor, so is $\grees P \rightarrow \grees M$. Thus a cofibrant replacement for $\grees M$ is given by $\grees P$. Yet $\varphi(\grees{A}) \iso A$, and the same is true for any summand of $A$, so $\varphi(\grees{P}) \iso P$ and the final $\Hom$-space is really just 
\begin{equation*}
\Hom_{\derivcat{\cat}}(\gamma P, \gamma N[i]) \iso \Hom_{\derivcat{\cat}}(\gamma M, \gamma N[i])=\Extfilt^i(M,N)
\end{equation*}
\end{proof}

\begin{rem}
The category $\cat=\filtmod{A}$ is not abelian, but it is in fact what Schneiders calls quasi-abelian \cite{qacs}. Quasi-abelian categories are so close to being abelian categories that nearly all of the tools from homological algebra carry through, not just derived functors. As we only need the $\Ext$-groups in $\cat$ for what follows, we decided to recharacterize this work in terms of model categories to keep the number of prerequisites down.
\end{rem}

\section{Rigidity of tilting modules}

\subsection{Tilting modules for quasi-hereditary algebras}

Let $A$ be a finite-di\-men\-sion\-al $K$-algebra. We recall the notion of a quasi-hereditary algebra. Suppose the irreducible $A$-modules $L(\lambda)$ are indexed by a poset $\Lambda$. Let $P(\lambda)$ and $I(\lambda)$ denote the projective cover and injective hull of $L(\lambda)$ respectively. Let $\weyl(\lambda)$ be the maximal quotient of $P(\lambda)$ whose composition factors are among $\{L(\mu) \mid \mu \leq \lambda\}$. These are the Weyl or standard modules. Define $\dweyl(\lambda)$ (the good or costandard modules) dually. We say that $A$ is quasi-hereditary if for all $\lambda \in \Lambda$
\begin{enumerate}[label=(\roman*)]
\item $\End_A \weyl(\lambda) \iso k$,

\item $P(\lambda)$ has a $\weyl$-filtration, i.e.~there is a series of submodules
\begin{equation*}
0=P_0<P_1<P_2<\dotsb<P_n=P(\lambda)
\end{equation*}
with $P_k/P_{k-1} \iso \weyl(\lambda_k)$ for some $\lambda_k \in \Lambda$.
\end{enumerate}
For graded quasi-hereditary algebras, a $\weyl$-filtration uses grade shifted copies of Weyl modules.

In \cite{ringel} Ringel constructed tilting modules for a quasi-hereditary algebra $A$. There are several notions of tilting and cotilting modules throughout representation theory, but in the special case of quasi-hereditary algebras there is an elementary description. We summarize this characterization of tilting modules in the next theorem.

\begin{thm}
Let $A$ be a quasi-hereditary algebra. For each weight $\lambda \in \Lambda$, there exists a unique indecomposable module $T(\lambda)$ such that 
\begin{enumerate}[label=(\roman*)]
\item $T(\lambda)$ has both a $\weyl$-filtration and a $\dweyl$-filtration.

\item There is a unique embedding of $\weyl(\lambda)$ as a submodule of $T(\lambda)$ and a unique quotient of $T(\lambda)$ isomorphic to $\dweyl(\lambda)$.

\item If $L(\mu)$ is a composition factor of $T(\lambda)$ then $\mu \leq \lambda$.
\end{enumerate}
\end{thm}

In fact a module $M$ has a $\dweyl$-filtration if $\Ext^1(\weyl(\lambda),M)=0$ for all $\lambda \in \Lambda$. Similarly, $M$ has a $\weyl$-filtration if $\Ext^1(M,\dweyl(\lambda))=0$ for all $\lambda \in \Lambda$. For the rest of this section we will assume that $A$ is a finite-dimensional quasi-hereditary algebra. We give $A$ a filtration structure using the radical series, as seen in Example \ref{exam:rad-filt}.

Suppose $M$ is an $A$-module with a $\weyl$-filtration $0=M_0<M_1< \dotsb < M_n=M$. Following \cite{bowman-martin} let $[\rad_s M: \head \weyl(\lambda)]$ denote the number of successive subquotients $M_{n_{s,i}}/M_{n_{s,i}-1}$ isomorphic to $\weyl(\lambda)$ such that $M_{n_{s,i}} \leq \rad^s M$ and such that there is a map $\rad^s M \rightarrow \weyl(\lambda)$ extending the quotient map $M_{n_{s,i}} \rightarrow \weyl(\lambda)$. We note that the value of $[\rad_s M: \head \weyl(\lambda)]$ does not depend on the choice of $\weyl$-filtration.

\begin{defn}
Let $M$ be an $A$-module. We say that $M$ has a radical-respecting $\weyl$-filtration if $M$ has a $\weyl$-filtration such that the homomorphisms $\rad^s M \rightarrow \weyl(\lambda)$ used to calculate $[\rad_s M : \weyl(\lambda)]$ induce isomorphisms $(\rad^{s+t} M \cap M_{n_{s,i}}+M_{n_{s,i}-1})/M_{n_{s,i}-1} \iso \rad^t \weyl(\lambda)$ for all $i$ and all $t \geq 0$.
\end{defn}

Varying $s$ and $i$, consider each $M_{n_{s,i}}/M_{n_{s,i}-1}$ as a subquotient of $\rad^s M$, which should be viewed as a module in its own right (i.e.~$J^m \rad^s M=\rad^{s+m} M$). The definition above is equivalent to saying that the isomorphisms carrying the subquotient $M_{n_{s,i}}/M_{n_{s,i}-1}$ to $\weyl(\lambda)$ are actually filtered isomorphisms. This implies that the Loewy layers of $M$ can be determined from the $\weyl$-filtration and the Loewy structure of the modules $\weyl(\lambda)$ using the following formula:
\begin{equation}
[\rad_s M: L(\mu)]=\sum_{\substack{t \leq s\\
\lambda \in \Lambda}} [\rad_t M : \head \weyl(\lambda)][\rad_{s-t} \weyl(\lambda) : L(\mu)] \label{eq:radresploewy}
\end{equation}

\begin{lem}
\label{lem:allradresp}
If a module $M$ has at least one radical-respecting $\weyl$-filtration, then all $\weyl$-fil\-tra\-tions are radical-respecting.
\end{lem}

\begin{proof}
Let $0=M_0<M_1< \dotsb < M_n=M$ be a $\weyl$-filtration. Say a subquotient $M_{k}/M_{k-1}$ isomorphic to $\weyl(\lambda_k)$ has a head on the $s_k$th radical layer of $M$, i.e.~the surjective quotient map $M_k \rightarrow \weyl(\lambda_k)$ extends to a map $\rad^{s_k} M \rightarrow \weyl(\lambda_k)$. Then for any $t \geq 0$, the restriction $\rad^{s_k+t} M \rightarrow \rad^t \weyl(\lambda_k)$ is still surjective. This shows that the composition factors from the $t$th radical layer of $\weyl(\lambda_k)$ occur at radical layer $h_{k,t} \geq s_k+t$. The $\weyl$-filtration is radical-respecting if $h_{k,t}=s_k+t$ in all such cases.

So suppose not, and pick $k$ and $t$ such that $s_k+t$ is minimal among those subquotients with $h_{k,t}>s_k+t$. By minimality the multiset of composition factors in the $(s_k+t)$th layer of $M$ must be subset of the multiset given by \eqref{eq:radresploewy}. Since at least one of these factors is missing from the $(s_k+t)$th layer, it must be a strict subset. But we already know that the Loewy series is given by \eqref{eq:radresploewy}, so this is impossible.
\end{proof}

\begin{prop}
\label{prop:grqhalg}
If the projective modules of $A$ have radical-respecting $\weyl$-fil\-tra\-tions, then $\grees{A}$ is graded quasi-hereditary.
\end{prop}
\begin{proof}
The projective modules for $\grees{A}$ are all of the form $\grees P(\lambda)$. The quotient map $P(\lambda) \rightarrow L(\lambda)$ is filtered surjective, so it is a fibration. As $\grees$ preserves fibrations we obtain a fibration of $\grees{A}$-modules, so $\grees L(\lambda)$ is a quotient of $\grees P(\lambda)$. It is clear that $\grees L(\lambda)$ is still irreducible as a $\grees{A}$-module, so this gives us both the irreducible $\grees{A}$-modules and their projective covers (up to grade shifting).

Let $0=P_0<P_1<\dotsb<P_n=P(\lambda)$ be a radical-respecting $\weyl$-filtration of $P(\lambda)$. As $A$ is quasi-hereditary, $P_n/P_{n-1} \iso \weyl(\lambda)$ and for $k<n$, $P_k/P_{k-1} \iso \weyl(\mu_k)$ and $\mu_k>\lambda$. For each subquotient $P_k/P_{k-1}$ there exists some $s_k$ such that as a filtered module $P_k/P_{k-1} \iso \weyl(\mu_k)$ when $P_k/P_{k-1}$ is viewed as a subquotient of $\rad^{s_k} P(\lambda)$. This means that when viewed as a subquotient of $P(\lambda)$, $P_k/P_{k-1} \iso \weyl(\mu_k)\langle s_k\rangle$.

The $\grees$ functor induces a chain of submodules $0=\grees P_0<\grees P_1<\dotsb<\grees P_n=\grees P(\lambda)$. In fact the subquotients in this filtration are isomorphic to $\grees \weyl(\mu)[s]$ for various $\mu$ and $s$, because
\begin{equation*}
\frac{\grees P_k}{\grees P_{k-1}} \iso \grees P_k/P_{k-1} \iso \grees(\weyl(\mu_k)\langle s_k\rangle) \iso \grees{\weyl(\mu_k)}(s_k)
\end{equation*}
Thus $\grees A$ is graded quasi-hereditary.
\end{proof}

\begin{defn}
A Weyl-irreducible (or $\weyl$-$L$) subquotient of a module $M$ is a subquotient $M'/M''$ isomorphic to a non-trivial extension of a module $W$ by $L(\mu)$, for some quotient $W$ of $\weyl(\lambda)$ and some weights $\lambda,\mu$ with $\mu>\lambda$. The subquotient $M'/M''$ is called a stretched subquotient if $M'$ is not isomorphic as a filtered module to a (possibly shifted) quotient of $P(\lambda)$.

An irreducible-good (or $L$-$\dweyl$) subquotient of a module $M$ is a subquotient $M'/M''$ isomorphic to a non-trivial extension of $L(\mu)$ by $U$, for some submodule $U$ of $\dweyl(\lambda)$ and some weights $\lambda,\mu$ with $\mu>\lambda$. The subquotient $M'/M''$ is called a stretched subquotient if $M'$ is not isomorphic as a filtered module to a (possibly shifted) submodule of $I(\lambda)$.
\end{defn}

\begin{thm}
\label{thm:greestilt}
Suppose $\grees A$ is quasi-hereditary. If a tilting module $T$ for $A$ has no stretched subquotients, then $\grees T$ is a tilting module for $\grees A$.
\end{thm}

\begin{proof}
Let $\lambda \in \Lambda$ be a weight. Consider a minimal filtered projective resolution for $\weyl(\lambda)$.
\begin{equation*}
\dotsb \rightarrow P_2 \rightarrow P_1 \rightarrow P(\lambda) \rightarrow \weyl(\lambda) \rightarrow 0
\end{equation*}
In particular $P_1$ is the direct sum of $P(\mu)\langle m \rangle$ ranging over $\mu,m$ such that $L(\mu)$ appears in the $m$th radical layer of $P(\lambda)$ and $\Ext^1(\weyl(\lambda),L(\mu)) \neq 0$. For $r \in \ZZ$ we will show that $\Ext^1(\weyl(\lambda),T\langle -r\rangle)=0$. We know that as an unfiltered module $\Ext^1(\weyl(\lambda),T)=0$ because $T$ is a tilting module. Let $f \in \Homradfilt(P_1,T\langle -r\rangle)$ be a non-zero cycle. The cycle $f$ can be viewed as an unfiltered homomorphism $\Omega(\weyl(\lambda)) \rightarrow T$, where 
\begin{equation*}
\Omega(\Delta(\lambda))=\ker(P(\lambda) \rightarrow \weyl(\lambda))
\end{equation*}
By the unfiltered $\Ext$-vanishing condition $f$ is the boundary of some unfiltered boundary $g \in \Hom(P(\lambda),T)$.

We claim that $g$ actually respects the filtrations. First, if $r<0$ there is nothing to prove, as 
\begin{equation*}
g(J^i P(\lambda))=g(\rad^i P(\lambda)) \subseteq \rad^i T \subseteq \rad^{i+r} T=J^i T\langle -r\rangle
\end{equation*}
So suppose $r\geq 0$. Choose $r'\geq r$ maximal such that $f \in \Homradfilt(P_1,T\langle -r'\rangle)$.

Let $M=\im g$ and $N=\im f=\im g|_{\Omega(\weyl(\lambda))}$. The submodule $M$ is a quotient of $P(\lambda)$ and $N$ is a submodule which is a quotient of $\Omega(\weyl(\lambda))$. So $g$ induces a surjective homomorphism between the quotients, as shown in the following diagram.
\begin{equation*}
\xymatrix{
0 \ar[r] & \Omega(\weyl(\lambda)) \ar[r] \ar[d]^{g|_{\Omega(\weyl(\lambda))}} & P(\lambda) \ar[r] \ar[d]^{g} & \weyl(\lambda) \ar[r] \ar[d] & 0 \\
0 \ar[r] & N \ar[r] \ar[d] & M \ar[r] \ar[d] & M/N \ar[r] \ar[d] & 0 \\
&  0 & 0 & 0 &
}
\end{equation*}
Thus $W=M/N$ is a quotient of $\weyl(\lambda)$. Let $0 \leq s \leq r'$ be maximal such that $M \subseteq \rad^s T$. In other words, the image of the head $L(\lambda)$ of $\weyl(\lambda)$ occurs in the $s$th radical layer of $T$. Pick an irreducible $L(\mu)$ appearing in $N/\rad N$ which is lowest in the radical series of $T$ and take a maximal submodule $N' \leq N$ such that $N/N' \iso L(\mu)$. Then $M/N'$ is a $\weyl$-$L$ subquotient of $T$. 

Since $N$ is also the image of $f$, it must be the case that the $L(\mu)$ factor is the head of some summand $P(\mu)\langle m \rangle$ of $P_1$, corresponding to a composition factor in the $m$th radical layer of $P(\lambda)$, with $m$ maximal. So $L(\mu)$ is in the $(r+m')$th radical layer of $T$, for some $m' \geq m$. If $s<r'$, then the filtration length of this subquotient is $r'+m'-s>m$, which is impossible as $m$ was chosen to be maximal and $T$ has no stretched subquotients. So $s=r'$, and thus 
\begin{equation*}
g(J^i P(\lambda))=g(\rad^i P(\lambda))=\rad^i g(P(\lambda)) \subseteq \rad^{r'+i} T \subseteq \rad^{r+i} T=J^i T\langle -r\rangle
\end{equation*}

This shows that $\Extradfilt^1(\weyl(\lambda),T\langle -r\rangle)=0$, so by applying the shift functor we have $\Extradfilt^1(\weyl(\lambda)\langle r \rangle,T)=0$. By Proposition \ref{prop:extfilt-extgr} this means that 
\begin{equation*}
\Extgr^1(\grees \weyl(\lambda)(r),\grees T)=0
\end{equation*}
As $\grees A$ is quasi-hereditary, this shows that $\grees T$ has a $\grees(\dweyl)$-filtration. A similar method shows that $\Extradfilt^1(T,\dweyl(\lambda)\langle r \rangle)=0$ so $\grees T$ also has a $\grees(\weyl)$-filtration, and hence it is a tilting module for $\grees A$.
\end{proof}

In particular when the above situation occurs $\grees T(\lambda)$ is the indecomposable $\grees A$ tilting module corresponding to $\lambda$, because $\grees$ preserves the multiplicities of $\weyl$-filtrations.

Another natural filtration that can be applied to modules is the socle filtration. For an $A$-module $M$, we can define a filtration $J^{\vee}$ by setting $J^{\vee (-i)} M=\soc^i M$ for $i\geq 0$ and $J^{\vee (-i)} M=0$ for $i<0$. It is easy to see that $M$ is a filtered $A$-module in this sense as well. Let $\grees^\vee$ denote the use of the $\grees$ functor using this alternative filtration.

\begin{thm}
\label{thm:rigidtilt}
Suppose $\grees{A}$ is quasi-hereditary. If an indecomposable tilting module $T=T(\lambda)$ for $A$ has no stretched subquotients for either the radical or the socle filtration, then $T$ is rigid. 
\end{thm}

\begin{proof}
Suppose $T=T(\lambda)$ is an indecomposable tilting module for $A$. If $T$ has no stretched subquotients, then by applying Theorem \ref{thm:greestilt} we know that $\grees T$ and $\grees^\vee T$ are both tilting modules for $\grees A$ corresponding to $\lambda$. But in a graded quasi-hereditary algebra there is only one such tilting module up to isomorphism and grade shifting. Since the gradings of $\grees T$ and $\grees^\vee T$ correspond to the radical and socle layers of $T$, this shows that $T$ has identical radical and socle layers.
\end{proof}

There is a partial converse to the above theorem.

\begin{cor}
\label{cor:rigidtiltconv}
Suppose $\grees A$ is quasi-hereditary. If $T=T(\lambda)$ is a rigid indecomposable tilting module for $A$ with radical-respecting $\weyl$- and $\dweyl$-filtrations, then $T$ has no stretched subquotients.
\end{cor}

\begin{proof}
From the proof of Proposition \ref{prop:grqhalg} $\grees{T}$ has $\grees(\weyl)$- and $\grees(\dweyl)$-fil\-tra\-tions. So $\grees{T}$ is a tilting module, and from the proof of Theorem \ref{thm:greestilt} any stretched subquotients would give rise to a non-vanishing $\Ext^1(\weyl(\lambda)\langle r\rangle,T)$ or $\Ext^1(T,\dweyl(\lambda)\langle r\rangle)$.
\end{proof}

\subsection{Duality of stretched subquotients}

The hypotheses of Theorems \ref{thm:greestilt} and \ref{thm:rigidtilt} are rather difficult to check in all but the most basic cases. In many applications $A$ has additional properties which can reduce this checking significantly.

\begin{cor}
\label{cor:greestiltdual}
Suppose $\grees A$ is quasi-hereditary. Let $T$ be a tilting module for $A$. If $T$ has a radical-respecting $\weyl$-filtration and has no stretched $\weyl$-$L$ subquotients, then $\grees T$ is a tilting module for $\grees A$.
\end{cor}

\begin{proof}
From the proof of Theorem \ref{thm:greestilt}, $\grees T$ has a $\grees(\dweyl)$-filtration. From the proof of Proposition \ref{prop:grqhalg}, $\grees T$ also has a $\grees(\weyl)$-filtration. Therefore $\grees T$ is tilting.
\end{proof}

The easiest way to show that $T$ has a radical-respecting $\weyl$-filtration is to show that $T$ has simple socle. For then $\head T \iso L(\lambda)$ for some $\lambda$, so $T$ is a quotient $P(\lambda)/U$ of $P(\lambda)$, which we assume already has a radical-respecting $\weyl$-filtration. As $T$ has a $\weyl$-filtration so does $U$ \cite[Theorem 3]{ringel}. Thus $\weyl$-filtrations of $T$ and $U$ give a $\weyl$-filtration of $P(\lambda)$, which is radical-respecting by Lemma \ref{lem:allradresp}. But the radical series of $T$ does not change from that of $P(\lambda)$, so $T$ also has a radical-respecting $\weyl$-filtration.

Another way to reduce the number of cases to check is to use duality. A duality functor on $\modcat{A}$ is a contravariant, additive, $K$-linear, exact functor $\delta:\modcat{A} \rightarrow \modcat{A}$ such that $\delta \circ \delta$ is naturally isomorphic to the identity. A BGG algebra is a quasi-hereditary algebra $A$ equipped with a duality functor $\delta$ which fixes irreducibles, i.e.~$\delta(L(\lambda)) \iso L(\lambda)$ for all $\lambda \in \Lambda$. In a BGG algebra we have $\delta(P(\lambda)) \iso I(\lambda)$ and $\delta(\weyl(\lambda)) \iso \dweyl(\lambda)$.

\begin{cor}
\label{cor:rigidtiltdual}
Suppose $A$ is a BGG algebra and $\grees A$ is quasi-hereditary. If $T=T(\lambda)$ is an indecomposable tilting module for $A$ such that $\grees T$ is a tilting module for $\grees A$ then $T$ is rigid.
\end{cor}

\begin{proof}
If $\grees T$ is a tilting module for $\grees A$, then $T$ has radical-respecting $\weyl$- and $\dweyl$-filtrations. Thus $\delta(T)$ has socle-respecting-respecting $\dweyl$- and $\weyl$-filtrations, so $\grees^\vee \delta(T)$ is also an indecomposable tilting module for $\grees A$. Yet $\delta(T) \iso T$, so $\grees^\vee \delta(T) \iso \grees^\vee T$. Proceed as in the proof of Theorem \ref{thm:rigidtilt}.
\end{proof}

Finally, there is a slightly simpler version of Corollary \ref{cor:rigidtiltconv} in the case of a BGG algebra.

\begin{cor}
Suppose $A$ is a BGG algebra and $\grees A$ is quasi-hereditary. If $T=T(\lambda)$ is a rigid indecomposable tilting module for $A$ with radical-respecting $\weyl$-filtration, then $T$ has no stretched subquotients.
\end{cor}

\begin{proof}
By duality $\delta(T) \iso T$ has a socle-respecting $\dweyl$-filtration. Yet $T$ is rigid, so $T$ actually has a radical-respecting $\dweyl$-filtration. Now use Corollary \ref{cor:rigidtiltconv}.
\end{proof}

\section{Eliminating stretched subquotients}

Finding and eliminating possible stretched subquotients in a module is in general extremely difficult. In addition to calculating the radical series of a module, one must also know enough about the submodule structure to figure out which subquotients exist. We describe some techniques for doing this, which we apply in the next section.

\subsection{Coefficient quivers}

Tilting modules corresponding to high weights tend to have complicated structure, with several composition factors interacting in intricate ways. One common method to depict the structure of a finite-length module is to use Alperin diagrams \cite{alperin}. However, often the necessary axioms for Alperin diagrams described in \cite{benson-carlson} do not hold in practice. As a result, the approach in the Appendix of \cite{bdm-small} using coefficient quivers must be used instead. Coefficient quivers can be viewed as a generalization of Alperin diagrams which always exist.

\begin{defn}
Let $Q=(Q_0,Q_1,s,t)$ be a quiver, and let $X=(X_i)_{i \in Q_0}$ be a representation of $Q$ over a field $K$. Suppose $\mathcal{B}$ is a basis for $X$ as a quiver representation, i.e.~$\mathcal{B}$ is a union of bases for each vector space $X_i$. The coefficient quiver of $X$ with respect to $\mathcal{B}$ is denoted $\Gamma(X,\mathcal{B})$. It has vertices indexed by $\mathcal{B}$. For $b \in \mathcal{B} \cap X_i$, $b' \in \mathcal{B} \cap X_j$ there is an arrow $b \rightarrow b'$ in $\Gamma(X,\mathcal{B})$ if and only if there is an arrow $\rho:i \rightarrow j$ such that the corresponding matrix entry $(X_\rho)_{bb'}$ is non-zero.
\end{defn}

Drawing a coefficient quiver can be thought of as ``unlacing'' the representation $X$ into its $1$-di\-men\-sion\-al irreducible composition factors. For a general module $M$ over some finite-di\-men\-sion\-al algebra $A$, Gabriel's theorem \cite[Proposition 4.1.7]{benson} is used to replace $A$ with a Morita equivalent quotient of $KQ$, where $Q$ is the $\Ext$-quiver of $A$. Thus the coefficient quiver of $M$ depends on the particular quotient and on the chosen basis. Like Alperin diagrams, coefficient quivers are conventionally drawn such that all arrows point downwards so that the arrowheads may be omitted. Another convention is that if $\Lambda$ is a labelling set for irreducibles $L(\lambda)$, we write $\lambda$ instead of $L(\lambda)$ in the coefficient quiver.

Arrow-closed subsets of a coefficient quiver $\Gamma$ for $M$ give submodules of $M$, and their complements give quotients. This describes much (but not all) of the submodule/quotient structure of $M$. For other submodules $M' \leq M$, it will be useful to describe which composition factors in $\Gamma$ correspond to composition factors of $M'$. Recall from linear algebra that we say a vector $v$ involves a basis vector $b$ if when $v$ is written as a linear combination of basis vectors, the coefficient corresponding to $b$ is non-zero. Since vertices of the coefficient quiver correspond to basis elements, we will say that a submodule $M'$ of $M$ involves a certain composition factor in $\Gamma$ if $M'$ contains a vector which involves the corresponding basis vector.

An Alperin diagram is called ``strong'' if both the radical series and the socle series can be calculated from the diagram \cite{alperin}. This concept can be extended to coefficient quivers as well. Although there exist modules which do not have strong coefficient quivers (e.g.~$T(4,3)$ in \cite[Appendix]{bdm-small}), for every module $M$ there exists a coefficient quiver which accurately depicts the radical series. In fact, for any subquotient there exists a coefficient quiver which will accurately depict the subquotient's radical series.

Stretched subquotients by necessity require ``stretched'' arrows connecting composition factors more than one radical layer apart. In most examples it will be impossible to draw a full coefficient quiver for a module. However, even knowing that certain arrows exist can be extremely helpful for eliminating stretched subquotients within tilting modules. We distinguish between two different kinds of arrows in a coefficient quiver:
\begin{itemize}
\item Solid lines ($\xymatrix@1{\lambda \ar@{-}[r] & \mu}$) denote arrows which definitely exist for the chosen basis.


\item Dotted lines ($\xymatrix@1{\lambda \ar@{.}[r] & \mu}$) denote arrows which may exist given certain values of the representing matrices $X_\rho$.
\end{itemize}


The following lemma shows that in many cases this requires multiple copies of a composition factor.

\begin{lem}
\label{lem:stretched-repeated-factors}
Let $M$ be a module with a radical-depicting coefficient quiver $\Gamma$. Suppose $\mu>\lambda$ are weights such that $L(\mu) \leq \rad_1 P(\lambda)$. Suppose further that some copy of $L(\lambda)$ in $M$ connects downward in $\Gamma$ to some factor $L(\lambda')$ which subsequently connects downward to a factor $L(\mu)$ with $\lambda'\nless \lambda$. Then $L(\lambda)$ is not involved in a stretched subquotient with this copy of $L(\mu)$ unless there is another copy of $L(\lambda')$ which connects downward from $L(\lambda)$ and downward to $L(\mu)$ or there is another copy of $L(\lambda)$ (possibly connected to $L(\mu)$) which connects downward to $L(\lambda')$.
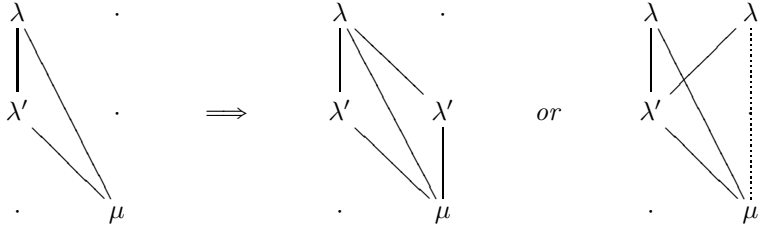
\begin{equation*}
\xymatrix{
\lambda \ar@{-}[d] \ar@{-}[ddr] & \cdot & & \lambda \ar@{-}[d] \ar@{-}[dr] \ar@{-}[ddr] & \cdot & & \lambda \ar@{-}[d] \ar@{-}[ddr] & \lambda \ar@{-}[dl] \ar@{.}[dd] \\
\lambda' \ar@{-}[dr] & \cdot & \Longrightarrow & \lambda' \ar@{-}[dr] & \lambda' \ar@{-}[d] & \text{or} & \lambda'\ar@{-}[dr] & \cdot \\
\cdot & \mu & & \cdot & \mu & & \cdot & \mu
}
\end{equation*}
\end{lem}
\begin{proof}
As $\lambda' \nless \lambda$, there is no composition factor $L(\lambda')$ within $\weyl(\lambda)$. If the given copy of $L(\lambda)$ connects to two copies of $L(\lambda')$, then we can change the basis for the $L(\lambda')$ vectors so that $L(\lambda)$ connects to one copy of $L(\lambda')$. In other words, we draw a new coefficient quiver
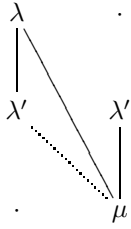
\begin{equation*}
\xymatrix{
\lambda \ar@{-}[d] \ar@{-}[ddr] & \cdot \\
\lambda' \ar@{.}[dr] & \lambda' \ar@{-}[d] \\
\cdot & \mu }
\end{equation*}
If both copies of $L(\lambda')$ connect downward to $L(\mu)$, then the proposed stretched subquotient is impossible. Thus the dotted arrow must not exist, so in particular in the original coefficient quiver both copies of $L(\lambda')$ must connect to $L(\mu)$, giving the first case.

Now assume that $L(\lambda)$ connects to exactly one copy of $L(\lambda')$ which connects to $L(\mu)$. This copy of $L(\lambda)$ alone cannot be the head of a stretched subquotient, because there is no way to quotient out $L(\lambda')$ without losing $L(\mu)$ as well. So there must be another copy of $L(\lambda)$ connected to $L(\lambda')$, giving the second case.
\end{proof}

\subsection{Calculating Loewy series}

The following results of Bowman and Martin on BGG algebras are extremely useful for calculating the radical series of projective modules. They will be used frequently in the following section.

\begin{prop}[{\cite[Theorem 6]{bowman-martin}}]
\label{prop:bgg-proj-recip}
Let $A$ be a BGG algebra with poset $\Lambda$. For $\lambda,\mu \in \Lambda$ we have the following reciprocity:
\begin{equation*}
[\rad_s P(\mu):L(\lambda)]=[\rad_s P(\lambda):L(\mu)]
\end{equation*}
\end{prop}

\begin{prop}[{\cite[Corollary 7]{bowman-martin}}] \label{prop:bgg-recip}
Let $A$ be a BGG algebra with poset $\Lambda$. For weights $\lambda,\mu \in \Lambda$ we have
\begin{equation*}
[\rad_s P(\mu):\head \weyl(\lambda)]=[\rad_s \weyl(\lambda):L(\mu)]
\end{equation*}
\end{prop}

Finally, we will use the following proposition to calculate socles of tilting modules from their characters. Its proof follows from \cite[Proposition A2.2]{donkin1998q}.

\begin{prop}
\label{prop:homdim}
Let $A$ be a quasi-hereditary algebra with poset $\Lambda$, and suppose $M$ is a module with a $\weyl$-filtration and $N$ is a module with a $\dweyl$-filtration. Then
\begin{equation*}
\dim \Hom_A(M,N)=\sum_{\lambda \in \Lambda} [M:\weyl(\lambda)][N:\dweyl(\lambda)]
\end{equation*}
\end{prop}

\section{Restricted tilting modules for $SL_4(K)$}

\subsection{Notation}

Our main source on representations of algebraic groups is \cite[II.1-7]{jantzen}. Let $G=SL_4(K)$, where $K$ is an algebraically closed field of characteristic $p>0$. For a dominant weight $\lambda$ let $\weyl(\lambda)$ be the Weyl module of highest weight $\lambda$, $\dweyl(\lambda)$ its contravariant dual, and $L(\lambda)$ the simple head of $\weyl(\lambda)$. For any finite saturated set $\pi$ of dominant weights, the full subcategory of rational $G$-modules whose composition factors are indexed by weights in $\pi$ is equivalent to a module category $\modcat{S(\pi)}$, where $S(\pi)$ is a finite-dimensional algebra called a generalized Schur algebra \cite{donkin}. The algebra $S(\pi)$ is quasi-hereditary (in fact a BGG algebra) with standard and costandard modules $\weyl(\lambda)$ and $\dweyl(\lambda)$ respectively. When necessary we will deal with $S(\pi)$-modules instead of rational $G$-modules for a sufficiently large set $\pi$.

We fix a notation for the weights. The root system corresponding to $SL_4(K)$ is $A_3$. Let $\alpha_1,\alpha_2,\alpha_3$ be the simple roots (with $\langle \alpha_1,\alpha_3^\vee\rangle=0$), and let $\omega_1,\omega_2,\omega_3$ be the corresponding fundamental weights, which span the weight lattice $X$ of $A_3$. We will use the notation $(\lambda_1,\lambda_2,\lambda_3) \in \ZZ^3$ to refer to the weight $\lambda_1 \omega_1+\lambda_2\omega_2+\lambda_3\omega_3$. In this notation, we have $\alpha_1=(2,-1,0)$, $\alpha_2=(-1,2,-1)$, and $\alpha_3=(0,-1,2)$. The set of dominant weights is therefore $X^+=\{(\lambda_1,\lambda_2,\lambda_3) \mid \lambda_1,\lambda_2,\lambda_3 \geq 0\}$, which can be given a partial order via the dominance ordering.

Recall that the affine Weyl group $W_p=W \rtimes pX$ acts on the vector space $X \otimes_{\ZZ} \RR$ via the dot action, which can be divided into simplicial fundamental regions called alcoves. There are $6$ alcoves in the restricted region $X_1$, which we label $C_i$ for $i$ one of $1$, $2$, $3$, $3'$, $4$, or $5$ (see Figure \ref{fig:alcoves}). The two alcoves $3$ and $3'$ are related `by symmetry' in a similar fashion to the $SL_3$ case. In addition, there are alcoves adjacent to (or ``flanking'') $3$ and $3'$ called $\fl$ and $\fl'$. The generators of $W_p$ are denoted $s_0,s_1,s_2,s_3$ where $s_i$ is the reflection in $\alpha_i$ and $s_0$ is the reflection in the upper wall of alcove $1$.
\begin{figure}[h]
\label{fig:alcoves}
\begin{equation*}
\xymatrix{
& C_5 \ar@{-}[d]^{s_2} \ar@{--}[dl] \ar@{--}[dr] & \\
C_{\fl} \ar@{-}[d]_{s_2} & C_4 \ar@{-}[dl]^{s_3} \ar@{-}[dr]_{s_1} & C_{\fl'} \ar@{-}[d]^{s_2} \\
C_3 \ar@{-}[dr]_{s_1} & & C_{3'} \ar@{-}[dl]^{s_3} \\
& C_2 \ar@{-}[d]^{s_0} & \\
& C_1 &  
}
\end{equation*}
\caption{The dominance lattice for the labelled alcoves. Solid lines indicate adjacent alcoves, with walls labelled using the $W_p$-generators. Dashed lines indicate dominance without adjacency.}
\end{figure}
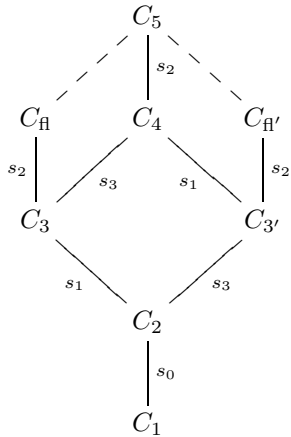

The linkage principle for algebraic groups states that if $L(\lambda)$ and $L(\lambda')$ are in the same block, then $\lambda' \in W_p \cdot \lambda$. If $V$ is a rational $G$-module, let $\pr_\lambda(V)$ denote the summand of $V$ whose composition factors have highest weights in $W_p \cdot \lambda$, and write $\mathcal{B}_\lambda$ for the full subcategory of modules such that $\pr_\lambda(V)=V$. For a dominant alcove $C$ and $\lambda,\mu \in \overline{C}$ the translation functor is defined by
\begin{equation*}
T_\lambda^\mu(V)=\pr_\mu\left(\pr_\lambda(V) \otimes L(w(\mu-\lambda))\right)
\end{equation*}
where $w \in W$ is chosen so that $w(\mu-\lambda) \in X^+$. Note that $T_\lambda^\mu$ is always exact as the composition of several exact functors. The translation principle states that $T_\lambda^\mu,T_\mu^\lambda:\mathcal{B}_\lambda \leftrightarrows \mathcal{B}_\mu$ are adjoint and mutually inverse if $\lambda$ and $\mu$ belong to the same set of alcoves. Therefore we can use alcove notation and write $L(1)$, $\weyl(1)$, etc.~when discussing general module structure without referring to specific weights.

Suppose $\lambda,\lambda' \in X^+$ belong to adjacent alcoves $C,C'$ with $\lambda<\lambda'$. Let $\mu$ be a weight on the wall between them, labelled by $s \in W$. The wall-crossing functor is defined to be $\theta_s=T_\mu^{\lambda'} \circ T_\lambda^\mu$, which is self-adjoint and exact. It is well-known that $\theta_s \weyl(\lambda)\iso \theta_s \weyl(\lambda')$, and we have the exact sequence
\begin{equation}
0 \rightarrow \weyl(\lambda') \rightarrow \theta_s \weyl(\lambda) \rightarrow \weyl(\lambda) \rightarrow 0 \label{eq:wall-crossing}
\end{equation}
We will use this exact sequence to calculate the character of $\theta_s(M)$ from the character of $M$.

Throughout this section we will use the notation $[L_0,L_1\dotsc,L_s]$ to depict the structure of the unique uniserial module $M$ with composition factors $L_0,\dotsc,L_s$ such that $\rad_i M \iso L_i$.

\subsection{The result}

From \cite{humphreys1976ordinary}, the character formulae of the labelled simple modules for type $A_3$ in terms of Weyl characters are fixed for $p$ sufficiently large. Alternatively, this fact can be viewed as a consequence of Lusztig's character formula for algebraic groups. We list these character formulae below.
\begin{align*}
[\weyl(1)]& =[L(1)] \\
[\weyl(2)]& =[L(2)]+[L(1)] \\
[\weyl(3)]& =[L(3)]+[L(2)] \\
[\weyl(\fl)]& =[L(\fl)]+[L(3)] \\
[\weyl(4)]& =[L(4)]+[L(3)]+[L(3')]+[L(2)]+[L(1)] \\
[\weyl(5)]& =[L(5)]+[L(4)]+[L(\fl)]+[L(\fl')]+[L(3)]+[L(3')]+[L(2)]
\end{align*}

Our goal in this section is to prove the following theorem.

\begin{thm}
\label{thm:A_3-general}

The regular restricted tilting modules for $G$ are all rigid. They have the following Loewy series and partial structure:
\begin{align*}
T(1)& =[1], & T(2)& =[1,2,1], \\
T(3)& =\begin{minipage}{34mm}
\def\objectstyle{\scriptstyle}\xymatrix@=6pt{
& 2 \ar@{-}[dl] \ar@{-}[dr] &  \\
3 \ar@{-}[dr] & & 1 \ar@{-}[dl] \\
& 2 & }
\end{minipage}, & T(\fl)& =\begin{minipage}{34mm}
\def\objectstyle{\scriptstyle}\xymatrix@=6pt{
& 3 \ar@{-}[dl] \ar@{-}[dr] &  \\
\fl \ar@{-}[dr] & & 2 \ar@{-}[dl] \\
& 3 & }
\end{minipage}, \\
T(4)& =\begin{minipage}{34mm}
\def\objectstyle{\scriptstyle}\xymatrix@=6pt{
& 2 \ar@{-}[d] & \\
3 \ar@{-}[d] & 1 & 3' \ar@{-}[d] \\
2 & 4 \ar@{-}[dr] \ar@{-}[d] \ar@{-}[dl] & 2 \\
3 \ar@{-}[dr] & 1 \ar@{-}[d] & 3' \ar@{-}[dl] \\
& 2 &}
\end{minipage}, & T(5)& =\begin{minipage}{34mm}
\def\objectstyle{\scriptstyle}\xymatrix@=6pt{
& & 3 \ar@{-}[d] & 3' \ar@{-}[d] & & \\
\fl' \ar@{-}[d] & \fl \ar@{-}[d] & 2 & 2 & 4 \ar@{-}[dr] \ar@{-}[d] \ar@{-}[dl] &  \\
3' & 3 & 5 \ar@{-}[dll] \ar@{-}[dl] \ar@{-}[d] \ar@{-}[dr] & 3 \ar@{-}[dr] & 1 \ar@{-}[d] & 3' \ar@{-}[dl] \\
\fl \ar@{-}[dr] & 2 \ar@{-}[d] \ar@{-}[dr] & 4 \ar@{-}[d] \ar@{-}[dl] & \fl' \ar@{-}[dl] & 2 & \\
& 3 & 3' & & & }
 \end{minipage}
\end{align*}
\end{thm}

The remainder of this section is devoted to the proof of this theorem.

\subsection{Weyl modules}

First we calculate the structure of the Weyl modules. We claim that the labelled Weyl modules have the following structure.
\begin{align*}
\weyl(1)& =[1], & \weyl(2)& =[2,1], & \weyl(3)& =[3,2], \\
\weyl(\fl)& =[\fl,3], & \weyl(4)& =\begin{minipage}{34mm}
\def\objectstyle{\scriptstyle}\xymatrix@=6pt{
& 4 \ar@{-}[dl] \ar@{-}[d] \ar@{-}[dr] &	\\
3 \ar@{-}[dr] & 1 \ar@{-}[d] & 3' \ar@{-}[dl] \\
& 2 & }
\end{minipage}, & \weyl(5)& =\begin{minipage}{34mm}
\def\objectstyle{\scriptstyle}\xymatrix@=6pt{
& & 5 \ar@{-}[dll] \ar@{-}[dl] \ar@{-}[dr] \ar@{-}[drr] & & \\
\fl \ar@{-}[dr] & 2 \ar@{-}[d] \ar@{-}[drr] & & 4 \ar@{-}[d] \ar@{-}[dll] & \fl' \ar@{-}[dl] \\
& 3 & & 3' & }
\end{minipage}
\end{align*}
The cases for $1,2,3,\fl$ are obvious from the character formulae. We proceed to cases $4$ and $5$.


If $L$ is a simple $G$-module, then from \eqref{eq:wall-crossing} we have
\begin{equation*}
\Hom_G(L,\weyl(4)) \leq \Hom_G(L,\theta_{s_3}\weyl(3)) \iso \Hom_G(\theta_{s_3}(L),\weyl(3))
\end{equation*}
and similarly for $\theta_{s_1}(L)$ and $\weyl(3')$. As $\theta_{s_3}L(1)$, $\theta_{s_3}L(3')$, and $\theta_{s_1}L(3)$ are all $0$, we must have $\soc \weyl(4)=L(2)$. The Lusztig character formula imposes a parity condition on the vanishing of the $\Ext^1$-groups, namely, $\Ext^1(L(\lambda),L(\mu))=0$ if the parity between $\lambda$ and $\mu$ (as measured by the length of a Weyl group element $w$ which sends $\lambda$ to $\mu$) is even \cite{scott1994quasihereditary}. As the remaining composition factors $L(3)$, $L(3')$, and $L(1)$ have the same parity, the structure of $\weyl(4)$ must be the one depicted above.

Similarly, for $L$ a simple $G$-module we have
\begin{equation*}
\Hom_G(L,\weyl(5))\leq \Hom_G(L,\theta_{s_2}\weyl(4)) \iso \Hom_G(\theta_{s_2}L,\weyl(4))
\end{equation*}
As $\theta_{s_2}L(\fl)$, $\theta_{s_2}L(\fl')$, and $\theta_{s_2}L(2)$ are all $0$ they cannot be summands of $\soc \weyl(5)$. From \eqref{eq:wall-crossing} we calculate
\begin{align*}
[\theta_{s_2}L(3)]& =[\theta_{s_2}\weyl(3)]-[\theta_{s_2}L(2)] \\
& =[\weyl(\fl)]+[\weyl(3)] \\
& =[L(\fl)]+2[L(3)]+[L(2)] \\
[\theta_{s_2}L(3')]& =[L(\fl')]+2[L(3')]+[L(2)] \\
[\theta_{s_2}L(4)]& =[\theta_{s_2}\weyl(4)]-[\theta_{s_2}L(3)]-[\theta_{s_2}L(3')]-[\theta_{s_2}L(2)]-[\theta_{s_2}L(1)] \\
& =[\weyl(5)]+[\weyl(4)]-[\theta_{s_2}L(3)]-[\theta_{s_2}L(3')] \\
& =[L(5)]+2[L(4)]+[L(1)]
\end{align*}
By considering the structure of $\weyl(4)$, $L(4)$ also is not contained in $\soc \weyl(5)$. So $\soc \weyl(5)$ contains at least one of $L(3)$ and $L(3')$, but by symmetry if it contains one it contains both, so $\soc \weyl(5)=L(3) \oplus L(3')$. Again, the remaining composition factors have the same parity so $\weyl(5)$ must have the structure depicted above.

\subsection{Projective modules}

The Loewy series and partial structures of the projective modules now follows using Propositions \ref{prop:bgg-proj-recip} and \ref{prop:bgg-recip}.
\begin{align*}
P(1)& =\begin{minipage}{34mm}
\def\objectstyle{\scriptstyle}\xymatrix@=8pt{
& 1 & & \\
2 \ar@{-} [d] & & 4 \ar@{-}[dl] \ar@{-}[d] \ar@{-}[dr] & \\
1 & 3 \ar@{-}[dr] & 1 \ar@{-}[d] & 3' \ar@{-}[dl] \\
& & 2 & }
\end{minipage}, & P(2)& =\begin{minipage}{34mm}
\def\objectstyle{\scriptstyle}\xymatrix@=8pt{
& & 2 \ar@{-}[d] & & & & \\
& 3 \ar@{-}[dl] & 1 & 5 \ar@{-}[dll] \ar@{-}[dl] \ar@{-}[d] \ar@{-}[dr] & 3' \ar@{-}[dr] & &  \\
2 & \fl \ar@{-}[dr] & 2 \ar@{-}[d] \ar@{-}[dr] & 4 \ar@{-}[d] \ar@{-}[dl] & \fl' \ar@{-}[dl] & 2 & 4 \ar@{-}[dll] \ar@{-}[dl] \ar@{-}[d] \\
& & 3 & 3' & 3 \ar@{-}[d] & 1 \ar@{-}[dl] & 3' \ar@{-}[dll] \\
& & & & 2 & &}
\end{minipage}, \\
P(3)& =\begin{minipage}{34mm}
\def\objectstyle{\scriptstyle}\xymatrix@=8pt{
& & 3 \ar@{-}[d] & & \\
& 4 \ar@{-}[dl] \ar@{-}[d] \ar@{-}[dr] & 2 & \fl \ar@{-}[dr] & \\
3 \ar@{-}[d] & 1 \ar@{-}[dl] & 3' \ar@{-}[dll] & 5 \ar@{-}[dll] \ar@{-}[dl] \ar@{-}[d] \ar@{-}[dr] & 3 \\
2 & \fl \ar@{-}[dr] & 2 \ar@{-}[d] \ar@{-}[dr] & 4 \ar@{-}[d] \ar@{-}[dl] & \fl' \ar@{-}[dl] \\
& & 3 & 3' & }
\end{minipage}, & P(\fl)& =\begin{minipage}{34mm}
\def\objectstyle{\scriptstyle}\xymatrix@=8pt{
& \fl \ar@{-}[d] & &  \\
& 3 & 5 \ar@{-}[dll] \ar@{-}[dl] \ar@{-}[d] \ar@{-}[dr] & \\
\fl \ar@{-}[dr] & 2 \ar@{-}[d] \ar@{-}[dr] & 4 \ar@{-}[d] \ar@{-}[dl] & \fl' \ar@{-}[dl] \\
& 3 & 3' & }
 \end{minipage}, \\
P(4)& =\begin{minipage}{34mm}
\def\objectstyle{\scriptstyle}\xymatrix@=8pt{
& 4 \ar@{-}[dl] \ar@{-}[d] \ar@{-}[dr] & & & \\
3 \ar@{-}[d] & 1 \ar@{-}[dl] & 3' \ar@{-}[dll] & 5 \ar@{-}[dll] \ar@{-}[dl] \ar@{-}[d] \ar@{-}[dr] & \\
2 & \fl \ar@{-}[dr] & 2 \ar@{-}[d] \ar@{-}[dr] & 4 \ar@{-}[d] \ar@{-}[dl] & \fl' \ar@{-}[dl] \\
& & 3 & 3' & }
\end{minipage}, & P(5)& =\begin{minipage}{34mm}
\def\objectstyle{\scriptstyle}\xymatrix@=8pt{
& & 5 \ar@{-}[dll] \ar@{-}[dl] \ar@{-}[dr] \ar@{-}[drr] & & \\
\fl \ar@{-}[dr] & 2 \ar@{-}[d] \ar@{-}[drr] & & 4 \ar@{-}[d] \ar@{-}[dll] & \fl' \ar@{-}[dl] \\
& 3 & & 3' & }
\end{minipage}
\end{align*}

It should be noted that Proposition \ref{prop:bgg-recip} only specifies where the heads of Weyl modules are located in the Loewy series. Any other composition factor in a Weyl subquotient must be located at least as far down in the radical series relative to the head of the subquotient as in the Weyl module itself. If none of the composition factors appear any further down, then \eqref{eq:radresploewy} holds for the Loewy series and the projectives have radical-respecting $\weyl$-filtrations, so $\grees{A}$ is a quasi-hereditary algebra by Proposition \ref{prop:grqhalg}.

There are several ways to show that \eqref{eq:radresploewy} holds. First of all, many possibilities can be ruled out using parity. For example, consider $P(1)$ and the factors $L(1)$, $L(3)$, and $L(3')$ inside $\weyl(4)$. These factors cannot occur any lower down the radical series, for this would require a connection (i.e.~a non-zero $\Ext^1$) between the $L(1)$ in $\weyl(2)$ and one of these modules, which is impossible by parity.

Secondly, we can use the fact that the projectives of the Schur algebra corresponding to a saturated subset of the weights are quotients of the projectives above. For example, consider $P(1)$ and the factor $L(1)$ inside $\weyl(2)$. We know that the projective cover of $L(1)$ for the Schur algebra corresponding to the weight set $\{1,2\}$ is a quotient of $P(1)$ by $\weyl(4)$. Therefore $\weyl(4)$ must be a submodule of $P(1)$, so in particular $L(1)$ cannot occur lower down in the radical series. This shows that $P(1)$ has the depicted Loewy series.

Finally, we can use Proposition \ref{prop:bgg-proj-recip} for any other cases which remain. For example, consider $P(2)$ and the factor $L(2)$ inside $\weyl(5)$. If $L(2)$ is lower down in the radical series, then it must be in the $4$th layer by parity. This would push $L(3)$ and $L(3')$ down to the $5$th layer, so $[\rad_5 P(2):L(3)]>0$. This implies that $[\rad_5 P(3):L(2)]>0$. But this is impossible (for the reasons above). Thus $L(2)$ (and similarly $L(4)$, $L(\fl)$, and $L(\fl')$) are actually in the $3$rd layer as depicted above.

\subsection{Tilting modules}

Now we proceed to prove the rigidity of the labelled tilting modules. Since all the weights we are dealing with are in the lowest $p^2$-alcove, we can calculate the characters of these tilting modules using a result of Soergel \cite{soergel-KL,soergel-KM}. The tilting characters and the known Weyl module structures give the socles of the tilting modules using Proposition \ref{prop:homdim}. In fact for all the labelled tilting modules we have $\soc T(\lambda)=\soc \weyl(\lambda)$.

Obviously $T(1)=[1]$, and $T(2)$ is $P_\pi(1)$ for $\pi=\{1,2\}$. If $\soc T(3) \iso \soc \weyl(3) \iso L(2)$ then $\head T(3) \iso L(2)$, so $T(3)$ is a quotient of $P_\pi(2)$ for $\pi=\{1,2,3\}$. The only quotient which possibly contains $\weyl(3)$ as a submodule is all of $P_\pi(2)$, and in order for it to have a $\dweyl$-filtration there must be a connection between the $L(2)$ in $\weyl(3)$ and the $L(1)$ in $\weyl(2)$. The case for $T(\fl)$ is similar.

The case for $T(4)$ is more complicated. Assuming $\soc T(4) \iso \soc \weyl(4) \iso L(2)$ we must have $T(4)$ as a quotient of $P_\pi(2)$, where $\pi=\{1,2,3,3',4\}$. As $P_\pi(2)$ has a radical-respecting $\weyl$-filtration, $T(4)$ also has one, so we can apply Corollaries \ref{cor:greestiltdual} and \ref{cor:rigidtiltdual} if we can show $P_\pi(4)$ (and therefore $T(4)$) has no stretched $\weyl$-$L$ subquotients. The only possible stretched $\weyl$-$L$ subquotient is between the $L(1)$ in $\weyl(2)$ and the $L(2)$ in $\weyl(4)$. By Lemma \ref{lem:stretched-repeated-factors} this can only happen if there is no connection between this copy of $L(1)$ and $L(4)$. But in that case, $P_\pi(4)$ would not have a quotient isomorphic to $\dweyl(4)$, which must be the case using the structure of $\dweyl(4)$ and Proposition \ref{prop:homdim}. Thus $T(4)$ is rigid, so it must in fact be all of $P_\pi(4)$.

Now assume $\soc T(5) \iso \soc \weyl(5) \iso L(3) \oplus L(3')$. Thus $T(5)$ is a quotient of $P(3) \oplus P(3')$. The only possible stretched $\weyl$-$L$ subquotient in $P(3) \oplus P(3')$ is between a copy of $L(2)$ in radical layer $1$ and $L(3)$ in the bottom radical layer (or the symmetric counterpart between $L(2)$ and $L(3')$). First, if $L(3)$ inside $\weyl(\fl)$ does not connect downwards to anything, then $\soc(P(3) \oplus P(3'))$ is too large, and any quotient which eliminates this socle does not have a quotient isomorphic to a submodule of $\dweyl(5)$. Similarly the $L(2)$ inside $\weyl(4)$ must connect downwards to some factor.

We know that $L(2)$ is connected to this $L(3)$ by the structure of $T(\fl)$. Thus we are in the situation of Lemma \ref{lem:stretched-repeated-factors}. The only other copy of $L(2)$ is not attached to this copy of $L(3)$. Thus $L(2)$ must also connect to the $L(3)$ inside $\weyl(4)$, which connects downwards to another $L(2)$. But we know that the first copy of $L(3)$ doesn't attach to this $L(2)$, because $\weyl(\fl)$ is a submodule of $P_\pi(3)$ for $\pi=\{1,2,3,3',4,\fl\}$. Thus we do not have a stretched subquotient. This shows that $T(5)$ must be rigid, and so it must have the Loewy series given above as $P(3) \oplus P(3')$ doesn't have any other non-trivial rigid quotients.

\bibliographystyle{hplain}
\bibliography{rigidtiltingrefs_article}
\end{document}